\documentclass[11pt]{amsart}
\usepackage{latexsym,amsmath,amssymb,epsfig,amsthm}
\usepackage{graphicx}
\usepackage{tikz}
\usetikzlibrary{calc}
\usepackage[cal=boondoxo,scr=euler]{mathalfa}
\usepackage[enableskew,vcentermath]{youngtab}
\usepackage{hyperref}
\hypersetup{
    colorlinks = true,
    linkbordercolor = {white},
    linkcolor = {blue},
    anchorcolor = {black},
    citecolor = {blue},
    filecolor = {blue},
    menucolor = {blue},
    runcolor = {cyan},
    urlcolor = {black}
}
\usepackage[shortlabels]{enumitem} % permits enumerating using letters or other symbols

\usepackage[margin=1.25in]{geometry}
\newtheoremstyle{theorems}% <name>
{11pt}% <Space above>
{11pt}% <Space below>
{\itshape}% <Body font>
{}% <Indent amount>
{\bfseries}% <Theorem head font>
{}% <Punctuation after theorem head>
{.5em}% <Space after theorem headi>
{}% <Theorem head spec (can be left empty, meaning `normal')>

\theoremstyle{theorems}
\newtheorem{theorem}{Theorem}[section]

\newtheorem{lemma}[theorem]{Lemma}
\newtheorem{proposition}[theorem]{Proposition}

\newtheoremstyle{definition}% <name>
{11pt}% <Space above>
{11pt}% <Space below>
{}% <Body font>
{}% <Indent amount>
{\bfseries}% <Theorem head font>
{}% <Punctuation after theorem head>
{.5em}% <Space after theorem headi>
{}% <Theorem head spec (can be left empty, meaning `normal')>

\theoremstyle{definition}
\newtheorem{definition}[theorem]{Definition}
\newtheorem{conjecture}[theorem]{Conjecture}

\newtheorem{example}[theorem]{Example}
\newtheorem{remark}[theorem]{Remark}

\newcommand{\rk}{{\rm rk}}

\newcommand{\des}{{\rm des}}
\newcommand{\Des}{{\rm Des}}

\newcommand{\link}{{\rm link}}

\newcommand{\sd}{{\rm sd}}

\newcommand{\cC}{{\mathcal C}}

\newcommand{\fF}{{\mathcal F}}
\newcommand{\gG}{{\mathcal G}}
\newcommand{\iI}{{\mathcal I}}
\newcommand{\jJ}{{\mathcal J}}
\newcommand{\lL}{{\mathcal L}}

\newcommand{\fS}{{\mathfrak S}}

\newcommand{\ZZ}{{\mathbb Z}}

\newcommand{\M}{{\mathsf{M}}}
\newcommand{\U}{{\mathsf{U}}}
\newcommand{\uH}{\underline{\mathrm{H}}}
\renewcommand{\H}{\mathrm{H}}
\newcommand{\I}{{\mathrm{I}}}
\newcommand{\uDelta}{\underline{\Delta}}

\newcommand{\sm}{{\smallsetminus}}
\begin{document}
\title[Decompositions of augmented Bergman complexes]
{A convex ear decomposition of the augmented Bergman 
complex of a matroid}

\author{Christos~A.~Athanasiadis}
\address{Department of Mathematics\\
National and Kapodistrian University of Athens\\
Panepistimioupolis\\ 15784 Athens, Greece}
\email{caath@math.uoa.gr}

\author{Luis~Ferroni}
\address{Department of Mathematics, University of Pisa, Pisa, Italy}
\email{luis.ferroni@unipi.it}

\date{\today}
%\thanks{CA was funded by.}

\begin{abstract}
In recent work of Braden, Huh, Matherne, Proudfoot and Wang, a class of 
simplicial complexes associated to matroids, called augmented Bergman 
complexes, was introduced. The present article concerns the face 
enumeration of these complexes. We 
prove that the augmented Bergman complex of any matroid admits a convex 
ear decomposition and deduce that augmented Bergman complexes are doubly 
Cohen--Macaulay and that they have top-heavy $h$-vectors. We provide some 
formulas for computing the $h$-polynomials of these complexes and exhibit 
examples which show that, in general, they are neither log-concave 
nor unimodal. 
\end{abstract}

\maketitle

\section{Introduction}
\label{sec:intro}

Two well-studied objects that one may associate to a matroid $\M$ are the 
\emph{independence complex}, denoted by $\I_{\M}$, and the \emph{Bergman 
complex}, denoted by $\uDelta_{\M}$. Recently, these simplicial complexes 
have received considerable attention. Despite their simple definitions, 
several questions about their face enumeration have turned out to be very 
challenging. As a notable example, the recent breakthroughs by Br\"and\'en 
and Huh \cite{BH20} and by Anari, Liu, Oveis-Gharan and Vinzant 
\cite{ALOV24} proved a long-standing conjecture by Mason, claiming that 
the entries of the $f$-vector of $\I_{\M}$ form an \emph{ultra 
log-concave} sequence. Moreover, the $h$-polynomial of the Bergman complex 
$\uDelta_{\M}$ was conjectured to be real-rooted (therefore, also ultra 
log-concave) in recent work of the first named author and 
Kalampogia-Evangelinou~\cite{AK23}. The conjecture was verified there
in interesting special cases, such as partition lattices and their 
standard type B analogues (also known as braid matroids of types A and B,
respectively) and subspace lattices.

In the \emph{singular} Hodge theory of matroids, developed by Braden, Huh, 
Matherne, Proudfoot and Wang \cite{BHMPW1,BHMPW2}, an instrumental role is 
played by a third simplicial complex, called the \emph{augmented Bergman 
complex}. This simplicial complex `interpolates' between $\I_{\M}$ and 
$\uDelta_{\M}$ and contains both as subcomplexes. The focus of this paper 
is its face enumeration.

Throughout this paper $\M$ will denote a matroid on the ground set $E = 
[n] := \{1, 2,\dots,n\}$, with lattice of flats $\lL(\M)$. We let $U = 
\{u_i : i \in E\}$ and $V_{\M} = \{v_F : F \in \lL(\M) \sm \{E\} \}$ be 
disjoint sets which are in one-to-one correspondence with $E$ and 
$\lL(\M) \sm \{E\}$, respectively.
\begin{definition} \label{def:augmented}
The \emph{augmented Bergman complex} of $\M$, denoted by $\Delta_{\M}$, 
is the simplicial complex on the ground set $U \cup V_{\M}$ whose faces 
are the sets of the form 
\begin{equation} \label{eq:Omega-faces}
\{u_i : i \in S \} \cup \{ v_{F_1}, v_{F_2},\dots,v_{F_k} \},
\end{equation}
where $S \in \I_{\M}$ is an independent set of $\M$, the $F_1, 
F_2,\dots,F_k \in \lL(\M) \sm \{E\}$ are proper flats and either 
$S \subseteq F_1 \subset F_2 \subset \cdots \subset F_k$ and $k \ge 1$, 
or $S \subseteq E$ and $k = 0$. 
\end{definition}

Augmented Bergman complexes were introduced in \cite{BHMPW1}. 
Their combinatorial and topological properties were studied in 
\cite{BKR22, AJ23}, where augmented Bergman complexes were shown to
be shellable and, in fact, vertex decomposable, respectively. In
particular, augmented Bergman complexes are Cohen--Macaulay (over 
$\ZZ$ and any field). As a well known consequence of either 
shellability or Cohen--Macaulayness (see Sections~II.3 and~III.2 
of \cite{Sta96}), their $h$-vectors have nonnegative entries.

The main contribution of this paper is a new list of inequalities 
that the $h$-vector of $\Delta_{\M}$ satisfies. The main technical 
tool employed is that of a convex ear decomposition (see 
Section~\ref{sec:ced}), a property that has been established 
for the independence complex $\I_{\M}$ by Chari \cite{Ch97} and for 
the Bergman complex $\uDelta_{\M}$ by Nyman and Swartz \cite{NS04}. 
This property, in addition to yielding inequalities for $h$-vectors, 
also establishes the double Cohen--Macaulayness of the complex: a 
simplicial complex $\Delta$ is said to be 
\emph{doubly Cohen--Macaulay} if it is Cohen--Macaulay and removing 
any vertex results in a subcomplex that 
is Cohen--Macaulay of the same dimension as $\Delta$ (see 
\cite[Section~III.3]{Sta96}).

\begin{theorem} \label{thm:main}
    The augmented Bergman complex $\Delta_{\M}$ has a 
    convex ear decomposition for every matroid $\M$. 
    In particular, $\Delta_{\M}$ is doubly Cohen--Macaulay,
    \begin{equation} \label{eq:h-vector-ineq-a}
    h_0(\Delta_{\M}) \le h_1(\Delta_{\M}) \le \cdots \le
    h_{\lfloor d/2 \rfloor}(\Delta_{\M})
    \end{equation}
    and
    \begin{equation} \label{eq:h-vector-ineq-b}
    h_i(\Delta_{\M}) \le h_{d-i}(\Delta_{\M})
    \end{equation}
    for $0 \le i \le d/2$, where $d$ is the rank of $\M$ and
		$(h_i(\Delta_{\M}))_{0 \le i \le d}$ is the $h$-vector 
		of $\Delta_{\M}$.
\end{theorem}

Inequalities \eqref{eq:h-vector-ineq-a} and \eqref{eq:h-vector-ineq-b}, 
often referred to as the \emph{top-heavy} inequalities for the $h$-vector 
of $\Delta_{\M}$, follow from the existence of a convex ear decomposition
\cite{Ch97}. They have been established for several fundamental 
enumerative invariants of matroids. For instance, as already mentioned, 
the independence complex (after removing coloops) and the Bergman complex 
of a matroid have convex ear decompositions. As a result, they have 
top-heavy (or strongly flawless, in other terminology) $h$-vectors and 
are doubly Cohen--Macaulay. The $h$-vector of the broken circuit 
complex, another simplicial complex of fundamental importance associated 
to a matroid, is top-heavy thanks to some deep results in the intersection 
theory of Chern--Schwartz--MacPherson cycles of matroids by Ardila, Denham 
and Huh \cite{ADH23}; see also \cite{KV16} for an earlier partial result. 
Furthermore, the sequence of Whitney numbers of the second kind of a 
matroid is top-heavy thanks to an important result by Braden, Huh, 
Matherne, Proudfoot and Wang \cite[Theorem~1.1]{BHMPW2}.

The content and structure of this paper is as follows. 
Theorem~\ref{thm:main} is proven in Section~\ref{sec:proof}; the proof 
may be viewed as an extension of the proof of the corresponding result 
of~\cite{NS04} for the Bergman complex. Section~\ref{sec:complexes} 
explains relevant definitions and background on simplicial complexes, 
and matroid complexes in particular. Moreover, it includes a technical 
lemma (see Lemma~\ref{lem:shellability-a}) which is crucial for the 
proof of Theorem~\ref{thm:main} and provides useful formulas for the 
$f$-vector and the $h$-vector of the augmented Bergman complex (see 
Propositions~\ref{prop:fOmega} and~\ref{prop:hOmega}). 
Section~\ref{sec:log-concavity} exhibits explicit counterexamples showing that top-heaviness cannot be upgraded to unimodality (or log-concavity).
Section~\ref{sec:chow} proves 
an augmented analogue of a formula conjectured by Hameister, Rao and 
Simpson \cite{HRS21} for the $h$-polynomial of the Bergman complex 
in the special case of uniform matroids.

\section{Simplicial complexes and matroids}
\label{sec:complexes}

This section includes background material on simplicial complexes, 
especially those related to matroids, as well as some preliminary 
results, which are essential in understanding Theorem~\ref{thm:main}
and its proof. We will assume familiarity with basic notions about 
simplicial complexes and matroids and suggest \cite{Bj92, Oxl11, 
Sta96} as general references.

\subsection{Face enumeration}
Let $\Delta$ be a $(d-1)$-dimensional simplicial complex and 
$f_{i-1}(\Delta)$ be the number of $(i-1)$-dimensional faces of 
$\Delta$, for $i \ge 0$. The sequence 
$(f_{-1}(\Delta),\ldots,f_{d-1}(\Delta))$ is called the 
\emph{$f$-vector} of $\Delta$. The \emph{$f$-polynomial} of 
$\Delta$ is defined as
\[ f(\Delta,x) = \sum_{F\in \Delta} x^{|F|} 
               = \sum_{i=0}^{d} f_{i-1} (\Delta) x^i. \]
Here we follow the usual convention that $f_{-1}(\Delta) = 1$ 
whenever $\Delta$ is nonempty and $f_{-1}(\Delta) = 0$ otherwise.

Correspondingly, the \emph{$h$-vector} $(h_0(\Delta), 
h_1(\Delta),\dots,h_d(\Delta))$ and the \emph{$h$-polynomial} 
$h(\Delta, x)$ of $\Delta$ are defined by a linear transformation 
of the $f$-vector:

\begin{align} \label{eq:hdef}
    h(\Delta, x) & = \sum_{F \in \Delta} x^{|F|} (1-x)^{d-|F|} 
		\\ \nonumber & = \sum_{i=0}^d f_{i-1} (\Delta) \, x^i (1-x)^{d-i} = 
		\sum_{i=0}^d h_i(\Delta) x^i.
\end{align}
Equivalently,
\begin{equation}\label{eq:f-from-h}
    f(\Delta,x) = \sum_{i=0}^d h_i(\Delta)\, x^i\, (1+x)^{d-i}.
\end{equation}
In other words, to obtain $f(\Delta,x)$ from $h(\Delta,x)$, 
one should first reverse (with respect to $d$) the coefficients of 
$h(\Delta,x)$, then substitute $x+1$ for $x$ and finally reverse 
the order of the coefficients again.

\begin{remark}
Some sources (for instance, \cite{ADH23, Bj92}) define the 
$f$-polynomial and the $h$-polynomial with their coefficients in 
reverse order. Although there is no ambiguity in the meaning of 
$f_i(\Delta)$ and $h_i(\Delta)$, the way in which the coefficients 
appear does matter when considering some properties of the 
polynomials, such as ultra log-concavity.
\end{remark}

\begin{example} \label{ex:one-dim}
Let $\Delta$ be the one-dimensional simplicial complex consisting 
of $\{a, b\}$, $\{b, c\}$, $\{c, d\}$, $\{d, e\}$, $\{a, e\}$, 
$\{b, e\}$ and their subsets, where $a, b, c, d, e$ are pairwise 
distinct (see Figure~\ref{fig:example}). Then, $f(\Delta,x) = 1 + 5x + 6x^2$ and hence 
$h(\Delta,x) = (1-x)^2 + 5x(1-x) + 6x^2 = 1 + 3x + 2x^2$ and 
$h(\Delta) = (1, 3, 2)$. 
 \begin{figure}[ht]
    \centering
    \begin{tikzpicture}
        [scale=0.85,auto=center,
         every node/.style={circle, fill=black, inner sep=1.7pt},
         every edge/.append style = {thick}]
        \tikzstyle{edges} = [thick];

        % --- Hexagon (regular, lado = 2) ---
        \node[label=left:$a$] (p3) at (-1.902113, -0.618034) {};
        \node[label=$b$] (p2) at (-3.077684,  1.000000) {};
        \node[label=$c$] (p1) at (-1.902113,  2.618034) {};
        \node[label=$d$] (a2) at (0,2) {};
        \node[label=right:$e$] (a1) at (0,0) {};
        \draw[edges] (a2) -- (p1);
        \draw[edges] (a1) -- (a2);
        \draw[edges] (p1) -- (p2);
        \draw[edges] (p2) -- (p3);
        \draw[edges] (p3) -- (a1);        \draw[edges] (p2) -- (a1);
    \end{tikzpicture}
    \caption{The simplicial complex of Example~\ref{ex:one-dim}.}
    \label{fig:example}
\end{figure}
\end{example}

\subsection{Shellability}
\label{sec:simplicial}

Let $\Delta$ be a $(d-1)$-dimensional simplicial complex and
let $\fF$ be the set of facets (maximal with respect to inclusion 
faces) of $\Delta$. Assume that $\Delta$ is pure, meaning that 
every facet has exactly $d$ elements. Then, $\Delta$ is said to be 
\emph{shellable} if $\fF$ can be linearly ordered so that the 
following condition holds: Given $F, G \in \fF$ with $F$ preceding 
$G$ in this order, there exists $F' \in \fF$ which precedes $G$ 
such that $F \cap G \subseteq F' \cap G$ and $|F' \cap G| = d-1$. 
Such a linear ordering is called a \emph{shelling} of $\Delta$.
Shellability has strong consequences on the combinatorial and 
topological properties of $\Delta$. For instance, it implies that 
$\Delta$ has a nonnegative $h$-vector and that it is Cohen--Macaulay 
(over $\ZZ$ and any field); see \cite[Section~1]{Bj84}
\cite[Section~7.2]{Bj92} and \cite[Section~III.2]{Sta96} for more 
information about this important concept.

Let $V = \{v_1, v_2,\dots,v_d\}$ be a $d$-element set. Let $\sd_V$ 
denote the barycentric subdivision of the boundary complex of the 
simplex $2^V$. The faces of $\sd_V$ are all chains of nonempty 
proper subsets of $V$. Consider another $d$-element set 
$U = \{u_1,u_2,\ldots,u_d\}$ which is disjoint from $\sd_V$, and 
set $u_S := \{ u_i : i \in S \}$ for $S \subseteq [d]$ and 
$v_S := \{ v_i : i \in S \}$ for $S \subset [d]$.

Consider the simplicial complex $\Omega_V$ having vertex set 
$U \cup \{ v_S : S \subset [d] \}$ and faces given by the sets of 
the form 
\[ u_S \cup \{ v_{S_1}, v_{S_2},\dots,v_{S_k} \}, \]
where $S \subseteq S_1 \subset S_2 \subset \cdots \subset S_k 
\subset [d]$ and $k \ge 1$, or $S \subseteq [d]$ and $k = 0$. The 
complex $\Omega_V$ is combinatorially isomorphic to the boundary 
complex of a $d$-dimensional polytope known as the simplicial 
stellohedron \cite[Example~5]{BKR22} \cite[Section~10.4]{PRW08}.

More generally, given a pure $(d-2)$-dimensional subcomplex 
$\Gamma$ of $\sd_V$, we define $\Omega_V(\Gamma)$ as the 
subcomplex of $\Omega_V$ whose facets are the facets $F$ of 
$\Omega_V$ for which $F \sm (U \cup \{ v_\varnothing \})$ is a 
nonempty face of $\Gamma$. By definition, $\Omega_V(\Gamma)$ 
is a pure $(d-1)$-dimensional subcomplex of $\Omega_V \sm \{U\}$ 
and $\Omega_V(\Gamma) = \Omega_V \sm \{U\}$ if and only if 
$\Gamma = \partial(2^V)$. 

\begin{lemma} \label{lem:shellability-a}
For every $d$-element set $V$ and every shellable 
$(d-2)$-dimensional subcomplex $\Gamma$ of $\sd_V$, the 
simplicial complex $\Omega_V(\Gamma)$ is a proper shellable 
$(d-1)$-dimensional subcomplex of $\Omega_V$.
\end{lemma}

\begin{proof}
Let $\fF$ be the set of facets of $\Omega_V(\Gamma)$ and for 
$S \subset [d]$, let $\fF_S$ be the set of those facets 
$F \in \fF$ for which $F \cap U = u_S$. 

Given $S \subset [d]$, we consider a maximal chain in the 
poset of nonempty subsets of $S$, ordered by inclusion, and 
let $\cC_S \in \sd_V$ be the corresponding chain of proper 
subsets of $V$. Then, $\fF_S$ is in one-to-one correspondence 
with the set of facets of the link $\link_\Gamma(\cC_S)$, which 
is a shellable subcomplex of $\Gamma$ of codimension equal to 
$|S|$. For example, $\fF_\varnothing$ is in one-to-one 
correspondence with the set of facets of $\Gamma$ and 
$\fF_{ \{u_i\} }$ is in one-to-one correspondence with the 
set of facets of $\link_\Gamma(v_i)$ for every $i \in [d]$.

Let us linearly order the facets of $\Omega_V(\Gamma)$ 
subject to the following conditions:
\begin{itemize}
\item[(a)] 
For all $S, T \subset [d]$, the elements of $\fF_S$ appear 
before those of $\fF_T$ whenever $|S| < |T|$.

\item[(b)] 
For every $S \subset [d]$, the elements of $\fF_S$ appear 
relative to each other according to a shelling order of 
$\link_\Gamma(\cC_S)$.
\end{itemize}

We now show that such a linear order of $\fF$ is a shelling 
of $\Omega_V(\Gamma)$. Given $F, G \in \fF$ with $F$ preceding 
$G$, we need to verify that there exists $F' \in \fF$ which 
precedes $G$ such that $F \cap G \subseteq F' \cap G$ and 
$|F' \cap G| = d-1$. Indeed, this follows from condition (b) 
if $F, G \in \fF_S$ for some $S \subset [d]$. Otherwise, we 
have $F \in \fF_S$ and $G \in \fF_T$ for some $S, T \subset 
[d]$ with $|S| \le |T|$ and $S \ne T$. Then, there exists 
$i \in (T \sm S)$ and the set obtained from $G$ by removing 
$u_i$ and including $v_{T \sm \{i\}}$ (which is equal to 
$v_\varnothing$, if $T = \{i\}$) can serve as the desired 
facet $F'$. 
\end{proof}

\begin{example} 
(a) For $d=2$, $\Omega_V$ is the one-dimensional simplicial 
complex having facets $\{v_\varnothing, v_1\}$, 
$\{v_\varnothing, v_2\}$, $\{u_1, v_1\}$, $\{u_2, v_2\}$ 
and $\{u_1, u_2\}$, where we have written $v_i$ instead of 
$v_{\{i\}}$ for $i \in \{1, 2\}$.

(b) Adopting similar notation for $v_S$, let $d=3$ and 
$\Gamma$ be the subcomplex of $\sd_V$ with facets $\{v_1, 
v_{12}\}$ and $\{v_2, v_{12}\}$. Then, $\Omega_V(\Gamma)$
has facets $\{v_\varnothing, v_1, v_{12}\}$, 
$\{v_\varnothing, v_2, v_{12}\}$, $\{u_1, v_1, v_{12}\}$, 
$\{u_2, v_2, v_{12}\}$ and $\{u_1, u_2, v_{12}\}$. Any 
linear ordering of these facets which lists $\{v_\varnothing, 
v_1, v_{12}\}$ and $\{v_\varnothing, v_2, v_{12}\}$ first, 
in any relative order, followed by $\{u_1, v_1, v_{12}\}$ 
and $\{u_2, v_2, v_{12}\}$, in any relative order, followed
by $\{u_1, u_2, v_{12}\}$ is a shelling of 
$\Omega_V(\Gamma)$ which satisfies the conditions in the 
proof of Lemma~\ref{lem:shellability-a}.
\end{example}

\subsection{Convex ear decompositions} 
\label{sec:ced}

Let $\Delta$ be a $(d-1)$-dimensional simplicial complex. 
A \emph{convex ear decomposition} of $\Delta$ is a sequence 
$(\Delta_1, \Delta_2,\dots,\Delta_m)$ of subcomplexes of 
$\Delta$ such that:
\begin{itemize}
\item[{\rm (i)}] 
$\Delta = \bigcup_{i=1}^m \Delta_i$,
    
\item[{\rm (ii)}] 
$\Delta_1$ is isomorphic to the boundary complex of a 
$d$-dimensional simplicial polytope and, for $2 \le k \le m$, 
$\Delta_k$ is a $(d-1)$-dimensional simplicial ball which is
isomorphic to a proper subcomplex of the boundary complex of 
a $d$-dimensional simplicial polytope; and
	
\item[{\rm (iii)}]
$\Delta_k \cap \left( \bigcup_{i=1}^{k-1} \Delta_i \right) = 
 \partial \Delta_k$ for $2 \le k \le m$.
\end{itemize}

Convex ear decompositions were introduced by Chari \cite{Ch97} 
as a tool to prove the top-heavy inequalities for $h$-vectors 
of simplicial complexes \cite[Corollary~2]{Ch97}. Swartz 
\cite[Theorem~4.1]{Sw06} showed that every simplicial complex 
$\Delta$ which admits a convex ear decomposition is doubly 
Cohen--Macaulay, meaning that $\Delta$ is Cohen--Macaulay and 
that every subcomplex which is obtained by removing a vertex 
from $\Delta$ is Cohen--Macaulay, of the same dimension as 
$\Delta$. We refer to \cite[Section~III.3]{Sta96} for more 
information about this important class of complexes. 

\begin{example} 
Let $\Delta$ be as in Example~\ref{ex:one-dim}. There
is the convex ear decomposition $(\Delta_1, \Delta_2)$ of 
$\Delta$, where $\Delta_1$ is the one-dimensional sphere with 
facets $\{b, c\}$, $\{c, d\}$, $\{d, e\}$ and $\{b, e\}$ and 
$\Delta_2$ is the one-dimensional ball with facets 
$\{a, b\}$ and $\{a, e\}$. Alternatively, one can choose 
$\Delta_1$ to have all facets of $\Delta$ other than 
$\{b, e\}$ and $\Delta_2$ to have the single facet $\{b, e\}$.   
\end{example}

\begin{remark} \label{rem:ced}
Let $\Sigma_1, \Sigma_2,\dots,\Sigma_m$ be subcomplexes of a 
$(d-1)$-dimensional simplicial complex $\Delta$ such that:
\begin{itemize}
\item[{\rm (i)}] 
$\Delta = \bigcup_{i=1}^m \Sigma_i$,
    
\item[{\rm (ii)}] 
$\Sigma_i$ is isomorphic to the boundary complex of a 
$d$-dimensional simplicial polytope for every $i \in [m]$; 
and
	
\item[{\rm (iii)}]
$\Sigma_k \cap \left( \bigcup_{i=1}^{k-1} \Sigma_i \right)$ 
is a $(d-1)$-dimensional simplicial ball, for all $2 \le k 
\le m$.
\end{itemize}
Then, $\Delta$ admits a convex ear decomposition in which 
the role of $\Delta_1$ is played by $\Sigma_1$ and the role 
of $\Delta_k$ for $2 \le k \le m$ is played by the complex 
obtained from $\Sigma_k$ by removing the interior faces of 
the ball $\Sigma_k \cap \left( \bigcup_{i=1}^{k-1} \Sigma_i 
\right)$. 
\end{remark}

\subsection{Matroids}
\label{sec:matroid}

Let $\M$ be a matroid on the ground set $E$, having lattice of 
flats $\mathcal{L}(\M)$. The \emph{Bergman complex} of $\M$, 
denoted by $\uDelta_{\M}$, is defined as the order complex of 
the poset $\mathcal{L}(\M) \smallsetminus 
\{\widehat{0}, \widehat{1}\}$, meaning the simplicial complex 
consisting of all chains in that poset. Note that $\uDelta_{\M} 
= \{\varnothing\}$ if $\M$ has rank zero or one.

The following proposition allows one to compute the $f$-polynomial 
of the augmented Bergman complex in terms of the $f$-polynomials of 
Bergman complexes of contractions of $\M$.

\begin{proposition} \label{prop:fOmega}
Let $\M$ be a matroid with set of bases $\mathscr{B}(\M)$. Then, the 
$f$-polynomial of the augmented Bergman complex of $\M$ is given by 
the formula
\[ f(\Delta_{\M},x) = |\mathscr{B}(\M)|\, x^d + (x+1) \, 
\sum_{\substack{S\in \I_{\M} \\ S \notin \mathscr{B}(\M)}} x^{|S|} \, 
f(\uDelta_{\M/S},x), \]
where $d$ is the rank of $\M$ and $\M/S$ is the contraction of $\M$ 
by $S$.
\end{proposition}

\begin{proof}
By definition, a face of $\Delta_{\M}$ has the form $\{u_i: i \in S\} 
\cup \{v_{F_1},\ldots, v_{F_k}\}$, where $S$ is an independent set and 
$F_1 \subset \cdots \subset F_k$ is a (possibly empty) chain of proper 
flats containing $S$. After fixing the independent set $S$, we have to 
enumerate chains of flats in the contraction $\M/S$, but allowing the 
first element in the chain to be equal to the bottom element while 
disallowing the last element of the chain to be equal to the top element.
In other words, we are enumerating the faces of a cone over the Bergman 
complex $\uDelta_{\M/S}$. There is one special case in which the 
operation is not an actual coning: when $\widehat{0}=\widehat{1}$ in 
the poset $\mathcal{L}(\M/S)$, which happens if and only if $\M/S$ has 
rank zero or, equivalently, when $S$ is a basis of $\M$. Summing over 
all independent sets yields the formula of the statement. 
\end{proof}

\begin{example} \label{ex:uniform} \rm
Consider the uniform matroid $\U_{d,n}$ of rank $d$ on $n$ elements. 
The preceding formula yields
\[ f(\Delta_{\U_{d,n}},x) = \binom{n}{d} \, x^d + (x+1) \, 
   \sum_{j=0}^{d-1} \binom{n}{j}x^{j} \, f(\uDelta_{\U_{d-j,n-j}},x). \]
Notice that the expression on the right depends on the $f$-polynomials 
of Bergman complexes of uniform matroids. These can be computed, for 
example, using \cite[Proposition~8.10]{FS24}. In Section~\ref{sec:chow} 
we will establish a surprising connection between $f(\Delta_{\U_{d,n}},x)$ 
and the Hilbert--Poincar\'e series of augmented Chow rings of uniform 
matroids. 
\qed
\end{example}

The next result provides an alternative concrete formula for computing 
the $h$-polynomial of the augmented Bergman complex of a matroid $\M$. 
Specifically, it is given by a convolution of $h$-polynomials of 
independence complexes of restrictions of $\M$ and Bergman complexes 
of contractions of $\M$ by flats (as opposed to independent sets).

\begin{proposition} \label{prop:hOmega}
The $h$-polynomial of the augmented Bergman complex of a matroid $\M$ 
on the ground set $E$ is given by the formula
\begin{equation} \label{eq:hOmega}
h(\Delta_{\M}, x) = h(\I_{\M},x) + x
  \sum_{F \in \lL(\M) \sm \{E\}} h(\I_{\M|_F},x) 
	                              h(\uDelta_{\M/F},x)
\end{equation}
where, for $F \in \lL(\M)$, the matroids $\M|_F$ and $\M/F$ are the 
restriction of $\M$ to $F$ and the contraction of $\M$ by $F$, 
respectively.
\end{proposition}

\begin{proof}
Given a flat $F \in \lL(\M) \sm \{E\}$, let us denote by $\gG_F$ the 
set of faces (\ref{eq:Omega-faces}) of $\Delta_{\M}$ for which 
$k \ge 1$ and $F_1 = F$. Then, $\Delta_{\M}$ is the disjoint union 
of $\I_{\M}$ and the sets $\gG_F$. Therefore, by 
equation~(\ref{eq:hdef}),

\begin{align*} 
  h(\Delta_{\M}, x) & = \sum_{G \in \I_{\M}} x^{|G|} 
  (1-x)^{d-|G|} + \sum_{F \in \lL(\M) \sm \{E\}} 
	\sum_{G \in \gG_F} x^{|G|} (1-x)^{d-|G|} 
	\\ & = h(\I_{\M},x) + \sum_{F \in \lL(\M) \sm \{E\}} 
	\sum_{G \in \gG_F} x^{|G|} (1-x)^{d-|G|}. 
\end{align*}
Since the sets $\{u_i : i \in S \}$ with $S \subseteq F$ are exactly 
the faces of $\I_{\M|_F}$ and the chains of flats $F_2 \subset \cdots 
\subset F_k \subset E$ with $F \subset F_2$ are exactly the faces of 
$\uDelta_{\M/F}$, we conclude that 

\begin{align*} 
\sum_{G \in \gG_F} x^{|G|} (1-x)^{d-|G|} & = 
x \sum_{G' \in \I_{\M|_F}} x^{|G'|} (1-x)^{\rk(F)-|G'|} 
\sum_{G'' \in \uDelta_{\M/F}} x^{|G''|} 
            (1-x)^{d-1-\rk(F)-|G''|}
\\ & = x \sum_{F \in \lL(\M) \sm \{E\}} 
	       h(\I_{\M|_F},x) h(\uDelta_{\M/F},x)
\end{align*}
for every $F \in \lL(\M) \sm \{E\}$ and the proof follows.
\end{proof}

\begin{example} \label{ex:uniform-h} \rm
For the uniform matroid $\U_{d,n}$ of rank $d$ on $n$ elements, 
equation~(\ref{eq:hOmega}) yields that
\begin{align*} 
h(\Delta_{\U_{d,n}}, x) & = \sum_{k=0}^d \binom{n-d+k-1}{d} x^k 
 + x \sum_{k=0}^{d-1} \binom{n}{k} h(\uDelta_{\U_{d-k,n-k}},x) \\ 
	                      & = \sum_{k=0}^d \binom{n-d+k-1}{d} x^k 
 + x \sum_{k=0}^{d-1} \binom{n}{k} \sum_{w \in \fS_{n-k}: \, 
                      \Des(w) \subseteq [d-k-1]} x^{\des(w)},
\end{align*}
where $\Des(w)$ and $\des(w)$ stand for the set and number of 
descents of a permutation $w$, respectively; see 
\cite[Example~2.2]{ADK24} for the justification of the last 
equality. This formula agrees with 
the fact that $h(\Delta_{\U_{n,n}}, x)$ is equal to the $n$th 
binomial Eulerian polynomial \cite[Example~5]{BKR22} 
\cite[Section~10.4]{PRW08}. 
\end{example}

\section{Proof of Theorem~\ref{thm:main}}
\label{sec:proof}

This section proves Theorem~\ref{thm:main}. Let $\M$ be a 
matroid of rank $d$ on the ground set $E = [n]$, as in the 
introduction. 

Let us briefly recapitulate some notions that are needed for the 
proof of Theorem~\ref{thm:main}, following \cite[Section~7.3]{Bj92} 
and \cite[Section~4]{NS04}.

Fix the total order $1 < 2 < \cdots < n$ of the ground set $E=[n]$ 
of $\M$. Let us agree to write every $d$-subset 
$S=\{i_1,\ldots,i_d\}\subseteq E$ as an ordered $d$-tuple 
$[i_1,\ldots,i_d]$, where $i_1 < \cdots < i_d$. This induces the 
lexicographic order on the set of all $d$-subsets of $E$: we say 
that $[i_1,\ldots,i_d]$ \emph{lexicographically precedes} 
$[j_1,\ldots,j_d]$ if there is an index $\ell$ such that 
$i_{\ell}<j_{\ell}$ and $i_m=j_m$ for all $m<\ell$. A 
\emph{broken circuit} of $\M$ is a circuit of $\M$ with its 
minimum element removed. An \emph{nbc basis} of $\M$ is a 
basis of $\M$ that does not contain a broken circuit; observe 
that all such bases must contain the element $1$. 

Consider a facet $\mathscr{F}$ of the order complex of 
$\mathcal{L}(\M)$ or, in other words, a maximal chain of flats 
of $\M$. All such chains are of the form $\mathscr{F} = 
\{\widehat{0} = F_0 \subset F_1 \subset\cdots \subset F_{d-1} 
\subset F_d = \widehat{1}\}$ and one may consider the ordered 
$d$-tuple $(a_1,\ldots,a_d)$, where $a_i = \min\{ e \in E : 
\overline{F_{i-1}\cup \{e\}} = F_{i}\}$ for each $i\in [d]$. 
This $d$-tuple is called the \emph{minimal labeling} associated 
to $\mathscr{F}$. It is evident that $\{a_1,\ldots,a_d\}$ is a 
basis of $\M$. 
For $S \subseteq E$, we denote by $\overline{S}$ 
the flat of $\M$ spanned by $S$. 

Given a basis $B$ of $\M$, we 
denote by $\sd_B$ the simplicial complex with faces 
\[ \{ v_{\overline{S_1}}, 
      v_{\overline{S_2}},\dots,v_{\overline{S_k}} 
			\}, \]
where $\varnothing \subset S_1 \subset S_2 \subset \cdots 
\subset S_k \subset B$. Clearly, $\sd_B$ is a subcomplex of 
$\uDelta_{\M}$ which is combinatorially isomorphic to $\sd_V$ 
for any $d$-element set $V$. Let $B$ be a basis of $\M$. 
Every facet of $\sd_B$ or, equivalently, every 
chain $\varnothing = S_0 \subset S_1 \subset \cdots \subset 
S_{d-1} \subset S_d = B$, corresponds to a permutation 
$(a_1,\ldots,a_d)$ of $B$, where $S_i = S_{i-1}\cup \{a_i\}$ 
for $i\in [d]$.

The following result is obtained by relying on the techniques 
of Nyman and Swartz \cite[Section~4]{NS04}. 

\begin{proposition} \label{prop:NS04}
    Let $B_1, B_2,\dots,B_m$ be the bases of $\M$, 
    sorted in the lexicographic order. Consider  
    the subcomplex 
    \begin{equation} \label{eq:intersect-sd}
      \Gamma_k = \sd_{B_k} \cap 
    	\left( \, \bigcup_{i<k} \sd_{B_i} \right)
    \end{equation}
    of $\sd_{B_k}$ for $2 \le k \le m$. 
    \begin{itemize}
    \item[{\rm (a)}] 
    The complex $\Gamma_k$ is a shellable 
    $(d-2)$-dimensional ball, if $B_k$ is an nbc basis. 
    
    \item[{\rm (b)}] 
    Otherwise, $\Gamma_k = \sd_{B_k}$.
    \end{itemize}
\end{proposition}

\begin{proof}
    Suppose first that $B_k$ is \emph{not} nbc. Consider 
    any facet $\mathscr{F}$ of $\sd_{B_k}$, say with corresponding 
    permutation $(a_1, a_2,\dots,a_d)$ of $B_k$, and let
    $j \in [d]$ be the smallest index for which $\{a_1, 
    a_2,\dots,a_j\}$ does contain a broken circuit. Replacing 
    $a_j$ with the smallest element of that circuit 
    yields a permutation of a basis $B_i$ with $i<k$ 
    which also represents $\mathscr{F}$. This argument shows that 
    $\mathscr{F}$ is a facet of $\sd_{B_i}$ and verifies part (b).
    
    For $k \in [m]$, let $\Lambda_k$ be the subcomplex 
    of $\sd_{B_k}$ whose facets are the maximal chains 
    \[ \{ v_{\overline{S_1}}, 
          v_{\overline{S_2}},\dots,v_{\overline{S_{d-1}}} 
    			\} \]
    of $\lL(\M)$, where $\varnothing \subset S_1 \subset 
    S_2 \subset \cdots \subset S_{d-1} \subset B_k$, 
    whose minimal label coincides with the corresponding
    permutation of $B_k$. We claim that 
    \begin{equation} \label{eq:unions}
      \bigcup_{i=1}^k \sd_{B_i} = \bigcup_{i=1}^k 
    	\Lambda_i
    \end{equation}
    for every $k \in [m]$. Indeed, the $\supseteq$ 
    inclusion is trivial. For its reverse, let $\mathscr{F}$ be 
    a facet of $\sd_{B_k}$ which does not belong to 
    $\Lambda_k$ and let $(a_1, a_2,\dots,a_d)$ be the 
    corresponding permutation of $B_k$. Since this 
    permutation is not equal to the minimal labeling 
    of the maximal chain of $\lL(\M)$ corresponding to
    $\mathscr{F}$, there exists an index $j \in [d]$ and an 
    element $a_0 \in E$ with $a_0 < a_j$ such that 
    \[ \overline{ \{a_1,\dots,a_{j-1}, a_j\} } = 
		   \overline{ \{a_1,\dots,a_{j-1}, a_0\}}. \]
		%%\[ a_1 \vee \cdots \vee a_{j-1} \vee a_j = a_1 
    %%       \vee \cdots \vee a_{j-1} \vee a_0, \]
    %
    %%where $\vee$ stands for the join (least upper 
		%%bound) operation in $\lL(\M)$. 
		Replacing $a_j$ with $a_0$ yields a permutation 
    of a basis $B_i$ with $i<k$ which also represents 
    $\mathscr{F}$. This shows that $k \ge 2$ and that 
		$\mathscr{F}$ is a facet of $\bigcup_{i=1}^{k-1} 
		\sd_{B_i}$. Hence, the claim follows by induction on $k$. 
     
    Suppose now that $2 \le k \le m$ and that $B_k$ is 
    nbc. It was shown in \cite[Section~4]{NS04} (see 
    Propositions~4.4 and~4.5 there) that $\Lambda_k$ 
    is a $(d-2)$-dimensional shellable ball and that 
    \begin{equation} \label{eq:boundary}
      \Lambda_k \cap \left( \, \bigcup_{i<k} 
      \Lambda_i \right) = \partial \Lambda_k.
    \end{equation}
    Combining equations~(\ref{eq:intersect-sd}), 
    (\ref{eq:unions}) and~(\ref{eq:boundary}) shows 
    that $\Gamma_k$ is disjoint from the interior of 
    the ball $\Lambda_k$. On the other hand, every 
    facet of $\sd_{B_k}$ which does not 
    belong to $\Lambda_k$ must belong to the union 
    $\bigcup_{i<k} \Lambda_i = \bigcup_{i<k} \sd_{B_i}$
    and hence to $\Gamma_k$. It follows that 
    $\Gamma_k$ is the complementary ball to 
    $\Lambda_k$ within the sphere $\sd_{B_k}$.
    Moreover, \cite[Lemma~4.3]{NS04} implies that 
    the facets of $\Gamma_k$ correspond to a set of 
    permutations of $B_k$ which forms an order ideal 
    in the weak order on the group of permutations 
    of $B_k$ and hence $\Gamma_k$ is also shellable 
    by \cite[Theorem~2.1]{Bj84}. 
\end{proof}

Given a basis $B$ of $\M$, we denote by $\Omega_B$ 
the simplicial complex with faces
\[ \{u_i : i \in S\} \cup \{ v_{\overline{S_1}}, 
     v_{\overline{S_2}},\dots,v_{\overline{S_k}} 
	 \}, \]
where $S \subseteq S_1 \subset S_2 \subset \cdots 
\subset S_k \subset B$ and $k \ge 1$, or $S 
\subseteq B$ and $k = 0$. Clearly, $\Omega_B$ is 
a subcomplex of $\Delta_{\M}$ which is 
combinatorially isomorphic to $\Omega_V$ for any 
$d$-element set $V$. 

\begin{proof}[Proof of Theorem~\ref{thm:main}]
Let $B_1, B_2,\dots,B_m$ be the bases of $\M$, 
sorted in the lexicographic order. Then, the 
augmented Bergman complex
$\Delta_{\M}$ is the union of its subcomplexes 
$\Omega_{B_i}$ for $i \in [m]$, each of which is 
combinatorially isomorphic to the boundary complex 
of the $d$-dimensional simplicial stellohedron. 
By the relevant definitions, for $2 \le k \le m$ 
we have
\begin{equation} \label{eq:intersect}
  \Omega_{B_k} \cap \left( \, \bigcup_{i<k} 
  \Omega_{B_i} \right) = \Omega_{B_k} (\Gamma_k), 
\end{equation}
where $\Gamma_k$ is as in 
Proposition~\ref{prop:NS04}. Combined with that, 
Lemma~\ref{lem:shellability-a} implies that 
$\Omega_{B_k} (\Gamma_k)$ is a (shellable) 
$(d-1)$-dimensional ball. Hence, by Remark~\ref{rem:ced}, 
the sequence $(\Omega_{B_1}, \Omega_{B_2},\dots,\Omega_{B_m})$ 
of polytopal simplicial spheres induces a convex ear 
decomposition of $\Delta_{\M}$. 
\end{proof}

\section{On unimodality and log-concavity}
\label{sec:log-concavity}

This section shows that $f$- and $h$-vectors of 
augmented Bergman complexes are not well behaved 
with respect to unimodality and log-concavity.
Let $p(x) = a_0 + a_1x + \cdots + a_d x^d$ be a 
polynomial of degree at most $d$ with nonnegative 
coefficients. We recall that $p(x)$ is called 
\begin{itemize}
\item[$\bullet$] 
  \emph{unimodal}, if $a_0 \le a_1 \le \cdots \le 
	a_k \ge a_{k+1} \ge \cdots \ge a_d$ for some $k \in
	\{0, 1,\dots,d\}$,
\item[$\bullet$] 
  \emph{log-concave}, if $a_i^2 \ge a_{i-1}a_{i+1}$ 
	for every $i \in [d-1]$,
	\item[$\bullet$] 
  \emph{ultra log-concave}, of order $m \ge d$, if
	\[ \frac{a_i^2}{\binom{m}{i}^2} \ge \frac{a_{i-1}}
	   {\binom{m}{i-1}} \cdot \frac{a_{i+1}}
		 {\binom{m}{i+1}} \]
  for every $i \in [d-1]$,
\item[$\bullet$] 
  \emph{real-rooted}, if every root of $p(x)$ is real, 
	or $p(x) \equiv 0$.
\end{itemize}
This is a hierarchy of increasingly stronger properties, 
where the implication $\text{log-concave} \Longrightarrow 
\text{unimodal}$ requires that $p(x)$ has no internal zeros, 
meaning there are no indices $i < j < k$ such that $a_i a_k 
\ne 0$ and $a_j = 0$. We refer to surveys by Stanley 
\cite{stanley-unimodality} and Br\"and\'en 
\cite{Bra15} for more information about unimodality, 
log-concavity and real-rootedness. We also point to the work 
of Liggett \cite{Lig97} and Br\"and\'en and Huh \cite{BH20}
for ultra log-concavity and multivariate generalizations.

Let $\M$ be a matroid on $n$ elements. The independence
complex $\I_{\M}$ was shown to have an ultra log-concave 
$f$-polynomial of order $n$ in \cite{BH20,ALOV24}
and a log-concave $h$-polynomial in \cite{AHK18}. A
stronger property has been conjectured for the 
$h$-polynomial of the Bergman complex.

\begin{conjecture}[\cite{AK23}] \label{conj:AKE}
The $f$-polynomial (equivalently, the $h$-polynomial) of 
$\uDelta_{\M}$ is real-rooted for every matroid $\M$.
\end{conjecture}

This conjecture, if true, implies that both $f$- and 
$h$-polynomials of Bergman complexes are ultra log-concave 
of any order, and in particular log-concave. These weaker 
properties are also open. 

One may suspect that $f$- and $h$-polynomials of augmented 
Bergman complexes also satisfy some version of log-concavity.
The following example rules out that possibility.

\begin{example}
Consider the uniform matroid $\U_{4,189}$. The $f$-vector 
of its augmented Bergman complex can be computed explicitly 
via Example~\ref{ex:uniform}:
\[ f(\Delta_{\U_{4,189}}) = (1, 1125559, 11181114, 27774180, 
                             69213375).\]
This $f$-vector is not log-concave, because
\[ 27774180^2 = 771405074672400  < 773882636199750 = 11181114 
   \cdot 69213375. \]
This example implies that $h$-vectors of augmented Bergman 
complexes are not always log-concave. In fact, the situation 
is even worse, as these $h$-vectors can even fail to be 
unimodal. For the uniform matroid $\U_{5,83}$ we have
\[ h(\Delta_{\U_{5,83}}) = (1, 1933066, 28121900, 60710014, 
                            28680319, 29034396).\]
The reader may observe that $60710014 > 28680319 < 29034396$. 
We have not been able to find a non-unimodal $f$-vector.
\qed
\end{example}

\begin{remark}
We have shown that $\Delta_{\M}$ is doubly Cohen--Macaulay. 
The first named author and Tzanaki have posed the question 
\cite[Question~7.2]{AT21} whether the $h$-vector of any 
$(d-1)$-dimensional doubly Cohen--Macaulay simplicial 
complex $\Delta$ satisfies the inequalities
\begin{equation}\label{eq:ineq-2cm}
\frac{h_0(\Delta)}{h_{d}(\Delta)} \le \frac{h_1(\Delta)}
     {h_{d-1}(\Delta)} \le \frac{h_2(\Delta)}{h_{d-2}(\Delta)} 
		 \le \cdots \le \frac{h_d(\Delta)}{h_{0}(\Delta)}.
\end{equation}
These inequalities imply that $h_i(\Delta) \le h_{d-i}
(\Delta)$ for $0 \le i \le d/2$. We have not found any 
counterexamples among $h$-vectors of augmented Bergman 
complexes. 
\end{remark}

\section{A connection to augmented Chow polynomials}
\label{sec:chow}

Hameister, Rao and Simpson conjectured a surprising relation between 
the $h$-vectors of Bergman complexes and the Hilbert--Poincar\'e 
series of Chow rings of uniform matroids \cite[Conjecture~6.2]{HRS21}. 
Their conjecture was proved recently by Ferroni, Matherne, Stevens 
and Vecchi \cite[Corollary~3.17]{FMSV24}. Specifically, they 
established the following formula. 

\begin{theorem}[\cite{FMSV24}]\label{thm:hrs-identity}
The $h$-polynomial of the Bergman complex of $\U_{d,n}$ can be 
expressed as
\[ x^{d-1} h(\uDelta_{\U_{d,n}},x^{-1}) = \sum_{i=1}^d 
   \binom{n-i-1}{d-i} \, \uH_{\U_{i,n}}(x), \]
where $\uH_{\M}(x)$ stands for the Hilbert--Poincar\'e series of the 
Chow ring of $\M$.
\end{theorem}

This section provides an analogue of this theorem for augmented Bergman 
complexes. We emphasize that, a priori, there are no reasons to expect 
such a concise relation between these two classes of polynomials and it 
is surprising that it continues to hold true in the augmented setting.

We assume familiarity with basic notions about Chow rings of 
matroids and their Hilbert--Poincar\'e series; our notation here follows 
\cite{FMSV24} closely, and we refer to that paper for any undefined 
terminology. One important distinction between the Chow polynomial and 
the augmented Chow polynomial is their definition on matroids with loops. 
In general, the Chow polynomial is defined to be zero if the matroid has 
a loop, whereas the augmented Chow polynomial of a matroid with loops 
does not change if one removes the loops.

\begin{theorem}
The $h$-vector of the augmented Bergman complex of $\U_{d,n}$ can be 
expressed as
\[ x^{d} h(\Delta_{\U_{d,n}},x^{-1}) = \sum_{i=0}^d \binom{n-i-1}{d-i} \,
   \mathrm{H}_{\U_{i,n}}(x), \]
where $\mathrm{H}_{\M}(x)$ stands for the Hilbert--Poincar\'e series of the augmented 
Chow ring of $\M$.
\end{theorem}

\begin{proof}
We will employ the following formula, established in \cite[Theorem~1.3~(iii)]{FMSV24}:
\[ \H_{\M}(x) = \sum_{F\in\mathcal{L}(\M)} x^{\rk(F)} \, \uH_{\M / F}(x). \]
By applying this identity to uniform matroids, we have the following chain of 
equalities:
\begin{align*}
\sum_{i=0}^d \binom{n-i-1}{d-i} \H_{\U_{i,n}}(x) &= 
\sum_{i=0}^d \binom{n-i-1}{d-i} 
\left(x^i + \sum_{j=0}^{i-1}\binom{n}{j} x^j \,\uH_{\U_{i-j,n-j}}(x)\right) \\
        &= \sum_{i=0}^d \binom{n-i-1}{d-i}\,x^i + \sum_{i=0}^{d}
				\sum_{j=0}^{i-1}\binom{n-i-1}{d-i}\binom{n}{j}x^j\,\uH_{\U_{i-j,n-j}}.
\intertext{Changing the order of summation in the second term and reindexing,
           we get} \\
\sum_{i=0}^d \binom{n-i-1}{d-i} \H_{\U_{i,n}}(x) &= 
\sum_{i=0}^d \binom{n-i-1}{d-i}\,x^i + \sum_{j=0}^{d-1}\binom{n}{j}x^j
\sum_{i=j+1}^{d}\binom{n-i-1}{d-i}\,\uH_{\U_{i-j,n-j}} \\
        &=\sum_{i=0}^d \binom{n-i-1}{d-i}\,x^i + 
				\sum_{j=0}^{d-1}\binom{n}{j}x^j\sum_{i=1}^{d-j}
				\binom{n-j-i-1}{d-j-i}\,\uH_{\U_{i,n-j}}.
\intertext{We now apply the formula of Theorem~\ref{thm:hrs-identity} to get}
\sum_{i=0}^d \binom{n-i-1}{d-i} \H_{\U_{i,n}}(x) &= 
\sum_{i=0}^d \binom{n-i-1}{d-i}\,x^i + \sum_{j=0}^{d-1}\binom{n}{j}x^j \, 
\left(x^{d-j-1} h(\uDelta_{\U_{d-j,n-j}},x^{-1})\right)\\ &= 
\sum_{i=0}^d \binom{n-i-1}{d-i}\,x^i + \sum_{j=0}^{d-1}
\binom{n}{j} x^{d-1} h(\uDelta_{\U_{d-j,n-j}},x^{-1}).
\end{align*}

Finally, we reverse the order of the coefficients and compare to the 
identity of Example~\ref{ex:uniform-h} and the proof follows.
\end{proof}

\section*{Acknowledgments}

During the early stages of this project, Luis Ferroni was a member at 
the Institute for Advanced Study, 
supported by the Minerva Research Foundation. The authors 
wish to thank Petter Br\"and\'en and Ed Swartz for useful 
conversations, and an anonymous referee for detailed comments that 
improved the exposition of this paper.

\newcommand{\etalchar}[1]{$^{#1}$}
\providecommand{\bysame}{\leavevmode\hbox to3em{\hrulefill}\thinspace}
\providecommand{\MR}{\relax\ifhmode\unskip\space\fi MR }
\providecommand{\MRhref}[2]{%
  \href{http://www.ams.org/mathscinet-getitem?mr=#1}{#2}
}
\providecommand{\href}[2]{#2}

\end{document}